\documentclass[hidelinks,12pt,oneside]{amsart}
\usepackage{amssymb,amsmath,amsthm}
\usepackage{hyperref}
\usepackage[bottom=2.5cm,right=2.5cm,top=2.5cm,left=2.5cm]{geometry}
\newtheorem{thm}{\textbf{Theorem}}[section]
\newtheorem{pro}[thm]{\textbf{Proposition}}
\newtheorem{rem}[thm]{\textbf{Remark}}
\newtheorem{cor}[thm]{\textbf{Corollary}}

\newtheorem{cond}[thm]{\textbf{Condition}}
\theoremstyle{definition}
\newtheorem{lemma}[thm]{\textbf{Lemma}}
\newtheorem{exmp}[thm]{Example}
\numberwithin{equation}{section}
\title{Families of Young Functions and Limits of Orlicz Norms}
\author{Sullivan MacDonald and Scott Rodney}
\date{May 2023}
\address{\label{francis} Sullivan F. MacDonald, MSc student in the Dept. of Mathematics \& Statistics\\ McMaster University \\ Hamilton, ON L8S4L8, Canada}
\email{\href{mailto:macdos55@mcmaster.ca}{macdos55@mcmaster.ca}}
\address{\label{scott} Scott Rodney\\ Dept. of Mathematics, Physics and Geology \\ Cape Breton University \\ Sydney, NS B1Y3V3, Canada} 
\email{\href{mailto:scott_rodney@cbu.ca}{scott\_rodney@cbu.ca}}

\begin{document}

\thanks{\hyperref[scott]{S.~Rodney} was supported by the NSERC Discovery Grant Program. \hyperref[francis]{S.F.~MacDonald} was supported by the NSERC USRA Program and the Dept. of Mathematics \& Statistics at McMaster University}

\begin{abstract}
Given a $\sigma$-finite measure space \((X,\mu)\), a Young function \(\Phi\), and a one-parameter family of Young functions \(\{\Psi_q\}\), we find necessary and sufficient conditions for the associated Orlicz norms of any function \(f\in L^\Phi(X,\mu)\) to satisfy 
\[
    \lim_{q\rightarrow \infty}\|f\|_{L^{\Psi_q}(X,\mu)}=C\|f\|_{L^\infty(X,\mu)}.
\] 
The constant \(C\) is independent of \(f\) and depends only on the family \(\{\Psi_q\}\). Several examples of one-parameter families of Young functions satisfying our conditions are given, along with counterexamples when our conditions fail.
\end{abstract}

\maketitle

\section{Introduction}

It is a well-known result in classical analysis (see e.g. \cite{royden} and \cite{rudin}) that if \((X,\mu)\) is a measure space and if \(f\in L^r(X,\mu)\) for some \(r\geq 1\), then
\begin{equation}\label{basic}
    \lim_{p\rightarrow\infty}\|f\|_{L^p(X,\mu)}= \|f\|_{L^\infty(X,\mu)}.
\end{equation}
The authors of \cite{DCU-SR21} investigated a similar limiting property of Orlicz norms associated with a one-parameter family of Young functions \(\{\Psi_q\}\) defined by setting \(\Psi_q(t)=t^p\log(e-1+t)^q\), where \(p\geq 1\) is fixed and \(q\) can be any positive real number. They showed that if $f$ belongs to the Orlicz space $L^{\Psi_{q_0}}(X,\mu)$ for some $q_0>0$ then 
\begin{equation}\label{loglim}
    \lim_{q\rightarrow\infty}\|f\|_{L^{\Psi_q}(X,\mu)}=\|f\|_{L^\infty(X,\mu)}.
\end{equation}
Here, \(\|f\|_{L^{\Psi_q}(X,\mu)}\) denotes the Luxembourg norm of \(f\) with respect to $\Psi_q$ and it is given by
\[
    \|f\|_{L^{\Psi_q}(X,\mu)}=\inf\bigg\{\lambda>0:\int_X\Psi_q\bigg(\frac{|f|}{\lambda}\bigg)~d\mu\leq1\bigg\}.
\]
\indent Modifying the proof of \cite[Thm. 1]{DCU-SR21}, we show that \eqref{loglim} holds for any family of Young functions $\{\Psi_q\}$ that satisfies Condition \ref{a-admissible}, which we call \emph{\(\delta\)-admissibility}. Like the authors of \cite{DCU-SR21}, our efforts were motivated by an application in partial differential equations, where we sought to employ a Moser iterative scheme in Orlicz spaces to study regularity of weak solutions to Poisson's equation. Ultimately those results were achieved using alternative techniques in \cite{DCU-SR20}.

Nevertheless, our main result Theorem \ref{main} may be useful in the study of related problems. Furthermore, it illustrates a surprising relationship between Orlicz spaces defined with respect to $(X,\mu)$, and the space $L^\infty(X,\mu)$ of essentially bounded functions on \(X\).

Throughout this work we assume that \((X,\mu)\) is a positive measure space with \(\mu(X)>0\). To state our main result concisely, we begin by defining \(\delta\)-admissibility.

\begin{cond}\label{a-admissible} Given \(\delta>0\), a one-parameter family of Young functions \(\{\Psi_q\}\) is said to be \(\delta\)-admissible if \(\displaystyle\lim_{q\rightarrow\infty} \Psi_q(t) = \infty\) for \(t>\delta\), and for \(0<t<\delta\) one of the following holds:
\begin{itemize}
    \item[(i)] If \(\mu(X)=\infty\) then \(\displaystyle\lim_{q\rightarrow\infty} \Psi_q(t) =0\),
    \item[(ii)] If \(\mu(X)<\infty\) then \(\displaystyle\lim_{q\rightarrow\infty} \Psi_q(t)< \mu(X)^{-1}\).
\end{itemize}
\end{cond}

Perhaps the simplest \(\delta\)-admissible family is the \(1\)-admissible collection obtained by taking \(\Psi_q(t)=t^q\) for \(q\geq1\), and more examples of \(\delta\)-admissible families are discussed in Section \ref{examples}. 

Now we state our main result concerning these families of Young functions.

\begin{thm}\label{main} Let $(X,\mu)$ be a $\sigma$-finite measure space and let \(\{\Psi_q\}\) be a one-parameter family of Young functions. Let \(\Phi\) be another Young function such that for any \(k>0\), the composition
\begin{equation}\label{phicond}
    \frac{t}{\Psi_q^{-1}(\Phi(t))}
\end{equation}
is non-decreasing on the interval \([0,k]\) whenever \(q\) is sufficiently large. Then the identity
\begin{equation}\label{infty:1}
    \lim_{q\rightarrow \infty} \|f\|_{L^{\Psi_q}(X,\mu)} =\frac{1}{\delta}\|f\|_{L^\infty(X,\mu)}
\end{equation}
holds for every \(f\in L^\Phi(X,\mu)\) if and only if $\{\Psi_q\}$ is a  \(\delta\)-admissible family for some \(0<\delta<\infty\). 
\end{thm}

\begin{rem} \label{Mainremark}
\begin{enumerate}\item We note that $\sigma$-finiteness of $(X,\mu)$ is only required for the forward implication. Further, in many cases \eqref{phicond} is non-decreasing on all of  $[0,\infty)$ for large \(q\). For example, if $\Psi_q(t)=t^q$ and $\Phi(t) = t^r$ for $r\geq 1$, then \eqref{phicond} is non-decreasing on $[0,\infty)$ for $q\geq r$.\\

\item The distinction between cases (i) and (ii) in Condition \ref{a-admissible} is needed to prove necessity of \(\delta\)-admissibility for identity \eqref{infty:1}, but it is not needed for sufficiency. Indeed, if \(\mu(X)<\infty\) and \(0<t<\mu(X)^{-1}\) then one cannot select sets of large enough \(\mu\)-measure with which to compute the limit of \(\Psi_q^{-1}(t)\) as \(q\rightarrow\infty\) using \eqref{infty:1}. We also note that Condition \ref{a-admissible} is weaker than the closely related sufficient condition where $\displaystyle\lim_{q\rightarrow\infty}\Psi_q(t)=0$ in (i) and (ii).% We also note that one can recover estimate (ii) from \eqref{infty:1} via Proposition \ref{invlim}, but not the identity in (i). %The definition of \(\delta\)-admissibility used above was chosen since it characterizes the families \(\{\Psi_q\}\) for which \eqref{infty:1} holds. Though Condition \ref{a-admissible} is inconvenient at times, since it requires us to treat separate cases in many of our proofs, it allows us to present our main results in a straightforward manner.
\end{enumerate}
\end{rem}

Usually it is difficult to find the inverse of a given Young function in closed form. To make \eqref{phicond} easier to verify for given examples of \(\Phi\) and \(\Psi_q\), notice that \eqref{phicond} is non-decreasing on \([0,k]\) exactly when the following function is non-decreasing on \([0,\Phi(k)]\): %Thinking about potential fixes here -Francis
\[
    \frac{\Phi^{-1}(t)}{\Psi_q^{-1}(t)}.
\]
So, to check that the conditions of Theorem \ref{main} hold for a Young function \(\Phi\) and a given family \(\{\Psi_q\}\), it suffices to compute the inverse of \(\Psi_q\) for each $q$, and to know either \(\Phi\) or \(\Phi^{-1}\).

The theorem stated above implies the main result of \cite{DCU-SR21},  and we include a streamlined proof of this special case in Section \ref{log-sec}. Furthermore, we observe that if \(\Phi(t)=t\) then \eqref{phicond} is non-decreasing whenever \(\Psi_q\) is a Young function, and from Theorem \ref{main} we obtain the following result that involves no growth condition.

\begin{cor}\label{maincor} If $(X,\mu)$ is a $\sigma$-finite measure space and \(\{\Psi_q\}\) is \(\delta\)-admissible for \(0<\delta<\infty\), then \eqref{infty:1} holds for every $f\in L^1(X,\mu)$.
\end{cor}

The remainder of this paper is organized as follows. In Section \ref{Prelims} we establish preliminary results for Young functions and Orlicz spaces, and in Section \ref{examples} we discuss several examples of \(\delta\)-admissible families to show how Theorem \ref{main} can be applied. Section \ref{proof} is then devoted to the proof of our main result, and Section \ref{log-sec} examines the special case of log-bumps. We conclude with Section \ref{counters}, where we demonstrate that identity \eqref{infty:1} can fail if the family \(\{\Psi_q\}\) is not \(\delta\)-admissible for any finite \(\delta>0\).

\section{Preliminaries}\label{Prelims}

This section contains a brief introduction to Young functions and Orlicz spaces. Our discussion is largely expository, and for a complete treatment the reader is referred to \cite{rao-ren91}. To begin, a non-negative function \(\psi: [0,\infty)\rightarrow [0,\infty)\) is said to be a density if it is right continuous, non-decreasing, \(\psi(t)=0\) exactly when \(t=0\), and \(\psi(t) \rightarrow\infty\) as \(t\rightarrow\infty\).  Given a density \(\psi\), the associated function \(\Psi:[0,\infty)\rightarrow [0,\infty)\) defined by
\[
    \Psi(t) = \int_0^t \psi(s)ds
\]
is called a Young function.  For our purposes, the important functional properties of \(\Psi\) are that it is continuous, strictly increasing,  and convex on \((0,\infty)\). Moreover, it is clear that \(\Psi(0)=0\) and that \(\Psi(t)\rightarrow\infty\) as \(t\rightarrow\infty\). Since the function \(\Psi(t)=t\) has constant density it is not a Young function according to the definition above, however for our purposes it can often be treated as one.

Given a Young function \(\Psi\) and a measure space \((X,\mu)\), the Orlicz space \(L^\Psi(X,\mu)\) is defined as the collection of \(\mu\)-measurable functions \(f:X\rightarrow\mathbb{R}\) for which the Luxembourg norm 
\[
    \|f\|_{L^{\Psi}(X,\mu)} = \inf\bigg\{\lambda>0: \int_X \Psi\left( \frac{|f|}{\lambda}\right)d\mu \leq 1\bigg\}
\]
is finite. Equipped with this norm \(L^\Psi(X,\mu)\) is a Banach space; see \cite{rudin}. The Orlicz classes generalize the classical Lebesgue spaces, and it is easy to verify that \(\|\cdot\|_{L^p(X,\mu)}=\|\cdot\|_{L^\Psi(X,\mu)}\) when \(\Psi(t)=t^p\) for \(p\geq 1\). Orlicz spaces can also provide a finer scale of norms then \(L^p(X,\mu)\) in the following sense: if \(\mu(X)<\infty\) and \(\Psi_q(t)=t^p(1+\log(1+t))^q\) for \(p\geq 1\) and \(q>0\), then for any \(\varepsilon>0\) we have
\[
    L^{p+\varepsilon}(X,\mu)\subsetneq L^{\Psi_q}(X,\mu)\subsetneq L^p(X,\mu).
\]
These inclusions can be verified using H\"older's inequality, and their strictness follows from the examples constructed in \cite{addie}.

In the sections that follow we employ several properties of the Luxembourg norm which we now establish. The first is a version of Chebyshev's inequality on the Orlicz scale. Henceforth we use the notation \(\chi_S\) to denote the indicator function of a set \(S\subseteq X\).

\begin{thm}[Chebyshev's Inequality]\label{markymark}
For any \(\alpha\geq0\), a \(\mu\)-measurable function \(f:X\rightarrow\mathbb{R}\), and a Young function \(\Psi\), the following inequality holds:
\begin{equation}\label{ChebyshevOrlicz}
    \alpha\Psi^{-1}(\mu(\{x\in X:|f(x)|\geq\alpha\})^{-1})^{-1}\leq\|f\|_{L^\Psi(X,\mu)}.
\end{equation}
\end{thm}

\begin{proof}
First we establish a simpler form of \eqref{ChebyshevOrlicz} in the norm of \(L^1(X,\mu)\) using a standard argument. Fix \(\alpha>0\) and define \(f_\alpha=\alpha\chi_{\{|f|\geq\alpha\}}\) so that \(f_\alpha\leq |f|\) holds pointwise and 
\[
    \mu(\{|f|\geq\alpha\})=\mu(\{x\in X:|f(x)|\geq\alpha\})=\frac{1}{\alpha}\int_Xf_\alpha(x)d\mu\leq\frac{1}{\alpha}\int_X|f(x)|d\mu=\frac{1}{\alpha}\|f\|_{L^1(X,\mu)}.
\]
Next we replace \(\alpha\) with \(\beta=\alpha/\|f\|_\Psi\). Using that \(\Psi\) is strictly increasing, it follows from the inequality above that
\[
    \mu(\{|f|\geq\beta\|f\|_\Psi\})=\mu\bigg(\bigg\{\Psi\bigg(\frac{|f|}{\|f\|_{L^\Psi(X,\mu)}}\bigg)\geq\Psi(\beta)\bigg\}\bigg)\leq\frac{1}{\Psi(\beta)}\int_X\Psi\bigg(\frac{|f|}{\|f\|_{L^\Psi(X,\mu)}}\bigg)d\mu.
\]
It is a well-known property of the Luxembourg norm, established in many standard references (e.g. \cite{rao-ren91}), that \(\int_X\Psi(|f|/\|f\|_{L^\Psi(X,\mu)})d\mu\leq 1\). It follows that \(\mu(\{|f|\geq\beta\|f\|_{L^\Psi(X,\mu)}\})\leq \Psi(\beta)^{-1}\), and since \(\Psi^{-1}\) is increasing this implies that \(\Psi^{-1}(\mu(\{x\in X:|f|\geq\alpha\})^{-1})^{-1}\leq\beta^{-1}\). Writing \(\beta\) in terms of \(\alpha\) and \(f\) gives \eqref{ChebyshevOrlicz}.
\end{proof}

Equipped with this result, we can compute the Orlicz norms of indicator functions exactly.

\begin{cor}\label{chars}
Let \(S\) be a \(\mu\)-measurable subset of \(X\). Then \(\|\chi_S\|_{L^\Psi(X,\mu)}=\Psi^{-1}(\mu(S)^{-1})^{-1}.\)
\end{cor}

\begin{proof}
The estimate \(\Psi^{-1}(\mu(S)^{-1})^{-1}\leq \|\chi_S\|_\Psi\) follows at once from Chebyshev's inequality. For the reverse inequality, observe that
\[
    \int_{X}\Psi\bigg(\frac{\chi_S}{\Psi^{-1}(\mu(S)^{-1})^{-1}}\bigg)d\mu=\int_S\Psi\bigg(\frac{1}{\Psi^{-1}(\mu(S)^{-1})^{-1}}\bigg)d\mu=\int_S\frac{1}{\mu(S)}d\mu=1.
\]
By the definition of the Luxembourg norm, this implies that \(\|\chi_S\|_\Psi\leq\Psi^{-1}(\mu(S)^{-1})^{-1}\).
\end{proof}

In this paper we work with limits of Orlicz norms that are defined by a one-parameter family  of Young functions. Subject to appropriate growth conditions, these families have useful pointwise properties which we will exploit in the sections to follow. Our main condition on these families is the following generalization of Condition \ref{a-admissible}.

\begin{cond}\label{ab-admissible} Given \(\alpha\geq 0\) and \(\beta\geq \alpha\), a family of Young functions \(\{\Psi_q\}\) is said to be \((\alpha,\beta)\)-admissible if \(\displaystyle\lim_{q\rightarrow\infty} \Psi_q(t) = \infty\) for \(t>\beta\), and for \(0<t<\alpha\) one of the following holds:
\begin{itemize}
    \item[(i)] If \(\mu(X)=\infty\) then \(\displaystyle\lim_{q\rightarrow\infty} \Psi_q(t) =0\),
    \item[(ii)] If \(\mu(X)<\infty\) then \(\displaystyle\lim_{q\rightarrow\infty} \Psi_q(t) <\mu(X)^{-1}\).
\end{itemize}
\end{cond}

\begin{pro}\label{ab-prop}
Let \(\{\Psi_q\}\) be an \((\alpha,\beta)\)-admissible family. If \(\mu(X)=\infty\) and \(t>0\), or if \(\mu(X)<\infty\) and \(t\geq \mu(X)^{-1}\), then
\begin{equation}\label{ab-est}
    \alpha\leq \liminf_{q\rightarrow\infty}\Psi_q^{-1}(t)\leq \limsup_{q\rightarrow\infty}\Psi_q^{-1}(t)\leq \beta.    
\end{equation}
\end{pro}

\begin{proof}
If \(\alpha=0\) then the first inequality in \eqref{ab-est} holds trivially, so we assume that \(\alpha>0\).

Fix \(t>0\) if \(\mu(X)=\infty\) and \(t\geq \mu(X)^{-1}\) if \(\mu(X)<\infty\), and assume toward a contradiction that there exists \(\eta>0\) such that
\[
    \liminf_{q\rightarrow\infty}\Psi_q^{-1}(t)\leq\alpha-\eta.
\]
Given any \(0<\varepsilon<\eta\) then, there exists an increasing sequence $\{q_j\}$ such that \(q_j\rightarrow\infty\) and \(j\rightarrow\infty\) and \(\Psi_{q_j}^{-1}(t)< \alpha-\varepsilon\) for each $j$. Since each \(\Psi_q\) is strictly increasing for all $q>0$, we find that \(t< \Psi_{q_j}(\alpha-\varepsilon)\) for each $j$. Taking the limit as $j\rightarrow\infty$, we see from Condition \ref{ab-admissible} that \(0<t\leq 0\) if \(\mu(X)=\infty\), and \(\mu(X)^{-1}\leq t<\mu(X)^{-1}\) if $\mu(X)<\infty$. In any case this is a contradiction, meaning that
\[
    \alpha-\eta<\liminf_{q\rightarrow\infty}\Psi_q^{-1}(t).
\]
Since \(\eta>0\) was arbitrary, we get the first inequality in \eqref{ab-est}. The estimates for the limit supremum in \eqref{ab-est} follow in an identical fashion.
\end{proof}

If \(\alpha=\beta=\delta\) then Condition \ref{ab-admissible} is the same as \(\delta\)-admissibility, and Proposition \ref{ab-prop} gives
\[\displaystyle\lim_{q\rightarrow\infty} \Psi^{-1}_q(t)=\delta\]
for each $t>0$ when \(\mu(X)=\infty\), and for each $t\geq \mu(X)^{-1}$ when $0<\mu(X)<\infty$. In fact, the limit identity above is equivalent to \(\delta\)-admissibility.

\begin{pro}\label{invlim}
A family of Young functions \(\{\Psi_q\}\) is \(\delta\)-admissible if and only if \begin{eqnarray}\label{scottd-est}
    \lim_{q\rightarrow\infty}\Psi_q^{-1}(t)=\delta
\end{eqnarray}
holds for all \(t>0\) if \(\mu(X)=\infty\), and for all $t\geq\mu(X)^{-1}$ if $0<\mu(X)<\infty$.
\end{pro}

\begin{proof}
Proposition \ref{ab-prop} gives the forward implication, leaving us to prove that if \eqref{scottd-est} holds for \(t\) in the appropriate range then \(\{\Psi_q\}\) is \(\delta\)-admissible. Regardless of whether \(\mu(X)\) is finite or infinite, if \(\Psi_q(t)\leq M\) for some \(M>0\) and large \(q\), then \(t\leq\Psi_q^{-1}(M)\) and \eqref{scottd-est} gives
\[
    t\leq \lim_{q\rightarrow\infty}\Psi_q^{-1}(M)=\delta.
\]
Since $M$ was arbitrary, it follows in the contrapositive that if \(t>\delta\), then \(\displaystyle\lim_{q\rightarrow\infty}\Psi_q(t)=\infty\).

Assume now that $0<\mu(X)<\infty$ and suppose that \eqref{scottd-est} holds for all $t\geq \mu(X)^{-1}>0$. If \[
    \lim_{q\rightarrow\infty}\Psi_q(t)\geq \mu(X)^{-1},
\] 
then given \(0<\varepsilon<1\) we have \(\Psi_q(t)> \mu(X)^{-1}(1-\varepsilon)\) whenever \(q\) is sufficiently large. By convexity of \(\Psi_q\) it follows that
\[
    \mu(X)^{-1}\leq \frac{\Psi_q(t)}{1-\varepsilon}\leq \Psi_q\bigg(\frac{t}{1-\varepsilon}\bigg),
\]
and so  \(t\geq(1-\varepsilon)\Psi_q^{-1}(\mu(X)^{-1})\) whenever $q$ is sufficiently large. Using \eqref{scottd-est} we take the limit to find
\[
    t\geq(1-\varepsilon)\lim_{q\rightarrow\infty}\Psi_q^{-1}(\mu(X)^{-1})=(1-\varepsilon)\delta.
\]
Since \(\varepsilon>0\) was arbitrary, we conclude that \(t\geq\delta\). Thus,  \(\displaystyle\lim_{q\rightarrow\infty}\Psi_q(t)<\mu(X)^{-1}\) when \(t<\delta\).

On the other hand, if \(\mu(X)=\infty\) and \eqref{scottd-est} is satisfied for \(t>0\), and if \(\Psi_q(t)\geq \varepsilon\) for some \(\varepsilon>0\) and for all large \(q\), then 
\[
    t\geq \lim_{q\rightarrow\infty}\Psi_q^{-1}(\varepsilon)=\delta.
\]
Thus, if \(t<\delta\) then \(\displaystyle\lim_{q\rightarrow\infty}\Psi_q(t)\leq\varepsilon\), and since \(\varepsilon>0\) was arbitrary we have \(\displaystyle\lim_{q\rightarrow\infty}\Psi_q(t)=0\).
\end{proof}

\section{Examples}\label{examples}

There are many examples of \(\delta\)-admissible families of Young functions, and moreover they are often easy to construct. In this section we showcase some families to illustrate the utility of our main result, Theorem \ref{main}.

\begin{exmp}
If $\Psi_q(t) = t^q$ and $\Phi(t) = t^r$ for some $r\geq 1$, then \eqref{phicond} is non-decreasing for \(t>0\) whenever $q\geq r$. Moreover the family \(\{\Psi_q\}\) is 1-admissible, and an application of Theorem \ref{main} gives the well-known identity \eqref{basic}.
\end{exmp}

\begin{exmp}
If \(\Phi(t)=t^r\) for \(r\geq 1\) and \(\Psi_q(t)=t^p\log(e-1+t)^q\) for fixed \(p\geq 1\), then
\[
    \frac{\Phi^{-1}(\Psi_q(t))}{t}=\frac{t^\frac{p}{r}\log(e-1+t)^\frac{q}{r}}{t}
\]
fails the growth condition of Theorem \ref{main} when \(r>p\), and satisfies it when \(r\leq p\), regardless of the value of $q>0$. In the latter case, Theorem \ref{main} recovers identity \eqref{loglim} for \(f\in L^r(X,\mu)\).
\end{exmp}

\begin{exmp}\label{logbumpexample}
Given \(N\in\mathbb{N}\) and \(p\geq 1\), consider the family of \(N^\mathrm{th}\)-order iterated log-bumps
\[
    \Psi_{q}(t) = t^p\underbrace{\log\cdots\log}_{N\;\mathrm{times}}(c+t)^q,
\]
where \(c\) is chosen independent of \(q\) so that \(\Psi_q(1)=1\). This family is \(1\)-admissible, and a straightforward adaptation of the argument in Section \ref{log-sec} shows that \eqref{phicond} is non-decreasing on any interval of the form $[0,k]$ for \(f\in L^{\Psi_{q_0}}(X,\mu)\) whenever \(q>q_0>0\) is sufficiently large. Thus, Theorem \ref{main} applies to the Orlicz norms characterized by the $N^\mathrm{th}$ order iterated log-bumps above, giving
\[
    \lim_{q\rightarrow \infty} \|f\|_{L^{\Psi_q}(X,\mu)} =\|f\|_{L^\infty(X,\mu)}
\]
whenever \(f\in L^{\Psi_{q_0}}(X,\mu)\) for some \(q_0>0\). We emphasize that in this example, the convergence of the \(\|f\|_{L^{\Psi_q}(X,\mu)}\) norms to \(\|f\|_{L^{\infty}(X,\mu)}\) is independent of \(p\). Thus for identity \eqref{infty:1} to hold when \(\{\Psi_q\}\) is a family of iterated log-bumps, it is not necessary to assume that \(f\in L^{p+\varepsilon}(X,\mu)\) for any \(\varepsilon>0\). 
\end{exmp}

\begin{exmp}
For any fixed Young function \(\Phi\), one can obtain a \(1\)-admissible family using the structure of \(N^\mathrm{th}\)-order iterated log-bumps by defining
\[
    \Psi_q(t)=\Phi(t)\underbrace{\log\cdots\log}_{N\;\mathrm{times}}(c+t)^q,
\]
where \(c\) is chosen so that \(\Psi_q(1)=\Phi(1)\) for all \(q\). Indeed, the iterated log-bumps of the form
\begin{align}\label{addie}
    \Psi_q(t)=\bigg(t\prod_{j=1}^N\underbrace{\log\cdots\log}_{j\;\mathrm{times}}(c_j+t)\bigg)^p\underbrace{\log\cdots\log}_{N\;\mathrm{times}}(c_N+t)^q
\end{align}
are of this type for \(p\geq1\), provided that the constants \(c_1,\dots,c_N\) are chosen so that the value of \(\Psi_q(\delta)\) is independent of \(q\) for some \(\delta>0\). Once again Theorem \ref{main} applies to this family, allowing us to reproduce the limit in \cite[Theorem 6.1]{addie}. This result is proved in \cite{addie} by means of a modification of the techniques of \cite{DCU-SR21}, which rely on the properties of iterated logarithms.

As in the last example, a similar calculation to that employed in Section \ref{log-sec} shows that if \(\Phi=\Psi_{q_0}\) for \(\Psi_q\) as in \eqref{addie}, then \eqref{phicond} is non-decreasing on any bounded interval of the form \([0,k]\) for \(k>0\) whenever \(q>q_0\) is sufficiently large. Once again, we conclude that the \(L^{\Psi_q}(X,\mu)\) norms of \(f\) converge to \(\|f\|_{L^\infty(X,\mu)}\), provided \(f\in L^{\Psi_{q_0}}(X,\mu)\) for some \(q_0>0\).
\end{exmp}

There are many more families for which \(\delta\)-admissibility can be established, and the interested reader is encouraged to construct their own examples.

\section{Proof of Theorem \ref{main}}\label{proof}

Every Orlicz norm used in this section is defined with respect to a fixed measure space \((X,\mu)\), so we will always write \(\|\cdot\|_{L^\Psi(X,\mu)}=\|\cdot\|_\Psi\) and $\|\cdot\|_{L^\infty(X,\mu)}=\|\cdot \|_\infty$. Fix \(\delta>0\) and suppose that \(\{\Psi_q\}\) is a \(\delta\)-admissible family of Young functions. Identity \eqref{infty:1} is trivial if \(f\equiv 0\), and we will treat the case of unbounded \(f\) separately at the end. Thus, we begin by assuming that \(0<\|f\|_\infty<\infty\), and we note that it is enough to prove \eqref{infty:1} when \(\|f\|_\Phi=1\). Since \eqref{phicond} is non-decreasing on $[0,\|f\|_\infty]$ by hypothesis for \(q\) sufficiently large, we have 
\begin{equation}\label{proof:1}
    \| f\|_{\Psi_q}=\left\| \Psi_q^{-1}(\Phi(|f|))\frac{|f|}{\Psi_q^{-1}(\Phi(|f|))}\right\|_{\Psi_q}
    \leq \left\| \Psi_q^{-1}(\Phi(|f|))\right\|_{\Psi_q}\frac{\|f\|_\infty}{\Psi_q^{-1}(\Phi(\|f\|_\infty))}.
\end{equation}
Additionally, we see that \(\| \Psi_q^{-1}(\Phi(|f|))\|_{\Psi_q}\leq 1\) since by definition of the Luxembourg norm,
\[ 
    \int_X \Psi_q\left(\Psi_q^{-1}\left(\Phi(|f(x)|)\right)\right)d\mu=\int_X\Phi(|f|)d\mu\leq 1.
\]
Moreover, in the case $0<\mu(X)<\infty$ we have that 
\[
    \Phi(\|f\|_\infty)\geq \mu(X)^{-1}\int_X \Phi(|f|)d\mu = \mu(X)^{-1}.
\]
Equality holds above since \(\|f\|_\Phi=1\) and since \(f\) is bounded by assumption (see e.g. \cite[Eq. (3.13)]{trudinger}). Using these estimates, we find from Proposition \ref{invlim} that 
\[
    \limsup_{q\rightarrow\infty}\|f\|_{\Psi_q}\leq\|f\|_\infty\lim_{q\rightarrow\infty}\Psi_q^{-1}(\Phi(\|f\|_\infty))^{-1}=\frac{\|f\|_\infty}{\delta}.
\]

Next suppose that $0<\varepsilon<\|f\|_\infty$ and let \(S=\{x\in\Omega:|f(x)|\geq \|f\|_\infty-\varepsilon\}\). From the definition of the essential supremum and Chebyshev's inequality it follows at once that \(0<\mu(S)\leq \Phi((\|f\|_\infty-\varepsilon)^{-1})\), meaning that \(\mu(S)\) is finite and nonzero. Moreover, Chebyshev's inequality with $\alpha = \|f\|_\infty-\varepsilon$ also shows that 
\[
    (\|f\|_\infty-\varepsilon)\Psi_q^{-1}(\mu(S)^{-1})^{-1}\leq \|f\|_{\Psi_q}.
\]  
From Proposition \ref{invlim} it follows that \(\Psi_q^{-1}(\mu(S)^{-1})^{-1}\rightarrow\delta^{-1}\) as \(q\rightarrow\infty\), since $\mu(S)^{-1}\geq \mu(X)^{-1}$ when $0<\mu(X)<\infty$. As a result we find that 
\[
    \frac{\|f\|_\infty-\varepsilon}{\delta}\leq \displaystyle\liminf_{q\rightarrow\infty}\|f\|_{\Psi_q}.
\]
Since $\varepsilon>0$ was arbitrary, this gives \(\displaystyle\delta^{-1}\|f\|_\infty\leq \displaystyle\liminf_{q\rightarrow\infty}\|f\|_{\Psi_q}\), proving that \eqref{infty:1} holds.

In the case where $\|f\|_\infty = \infty$, choose $N>1$ and set $f_N = \min\{|f|,N\}$ so that $\|f_N\|_\infty=N$. Applying our work above, we see that \[\displaystyle\liminf_{q\rightarrow \infty} \|f\|_{\Psi_q} \geq \liminf_{q\rightarrow \infty} \|f_N\|_{\Psi_q}\geq \delta^{-1}\|f_N\|_\infty = \delta^{-1}N.\] 
Since $N$ may be chosen arbitrarily large, we find $\displaystyle\liminf_{q\rightarrow \infty} \|f\|_{\Psi_q}=\infty$ as required.

Now we show that if \eqref{infty:1} holds for all \(f\in L^\Phi(X,\mu)\), then the family $\{\Psi_q\}$ is $\delta$-admissible. Specifically, we utilize the characterization of \(\delta\)-admissible families given by Proposition \ref{invlim} to recognize that it is enough to show that
\[
    \lim_{q\rightarrow\infty} \Psi^{-1}_q(t) = \delta
\]
for every $t>0$ when $\mu(X)=\infty$, and for $t\geq \mu(X)^{-1}$ when \(\mu(X)\) is finite and positive.

Suppose first that $\mu(X)=\infty$.  Given $t>0$, we use that $(X,\mu)$ is $\sigma$-finite to select sets $S_1,S_2\subset X$ of sufficiently large measure so that with $t_j = \mu(S_j)^{-1}$ we have \(0<t_2<t_1<t\). Using \eqref{infty:1} with $f_j = \chi_{S_j}\in L^\Phi(X,\mu)$ we find from Corollary \ref{chars} that
\[
    \lim_{q\rightarrow \infty} \Psi_q^{-1}(t_j)=\lim_{q\rightarrow \infty} \|f_j\|_{L^{\Psi_q}(X,\mu)}^{-1} =\delta\|f_j\|_{L^\infty(X,\mu)}^{-1}=\delta.
\]
Since we may choose $\lambda\in(0,1)$ so that \(\lambda t_2+(1-\lambda)t = t_1\), the concavity of $\Psi_q^{-1}$ gives
\[
    \lambda\Psi_q^{-1}(t_2) + (1-\lambda)\Psi_q^{-1}(t) \leq \Psi_q^{-1}(t_1).
\]
Letting $q\rightarrow\infty$ we find after taking a limit supremum and rearranging that
\[
    \limsup_{q\rightarrow\infty}\Psi_q^{-1}(t) \leq \delta.
\]
More, since $\Psi_q^{-1}$ is increasing, $\delta\leq\displaystyle\liminf_{q\rightarrow\infty}\Psi_q^{-1}(t)$.  Thus, $\displaystyle\lim_{q\rightarrow\infty}\Psi_q^{-1}(t) = \delta$ for every $t>0$.

In the case that $0<\mu(X)<\infty$, if $t>\mu(X)^{-1}$ we may proceed exactly as above, see Remark \ref{Mainremark}.  If $t=\mu(X)^{-1}$, the required estimate follows at once by applying \eqref{infty:1} to $f =\chi_X\in L^\Phi(X,\mu)$. In any case we have established that 
\[
    \displaystyle\lim_{q\rightarrow\infty}\Psi_q^{-1}(t) = \delta
\]
for $t\geq \mu(X)^{-1}$ when \(0<\mu(X)<\infty\), and for \(t>0\) when \(\mu(X)=\infty\). It follows from Proposition \ref{invlim} that \(\{\Psi_q\}\) is \(\delta\)-admissible.\hfill\(\Box\)

\begin{rem}\label{ab-rem}
If we use the more general Condition \ref{ab-admissible} in place of Condition \ref{a-admissible} in Theorem \ref{main}, simple modifications of the proof above show that one has the estimates
\begin{equation}\label{ab}
    \frac{1}{\beta}\|f\|_\infty\leq \liminf_{q\rightarrow\infty}\|f\|_{\Psi_q}\leq \limsup_{q\rightarrow\infty}\|f\|_{\Psi_q}\leq \frac{1}{\alpha}\|f\|_\infty.
\end{equation}
\end{rem}

It may be the case that the limit of the norms \(\| f\|_{\Psi_q}\) does not exist for a family \(\{\Psi_q\}\) which is \((\alpha,\beta)\)-admissible, as we show with the following example. Let
\[
    \Psi_q(t)=\begin{cases}
    \hfil\frac{1}{2}t^q & 0\leq t\leq \frac{1}{2},\\
    \frac{1}{2}(t^q+(2t-1)^{2+\sin q}) & \frac{1}{2}<t< 1,\\
    \hfil\frac{1}{2}(t^q+(2t-1)^3) & \hfill t\geq1,
    \end{cases}
\]
so that \(\{\Psi_q\}\) is \((\frac{1}{2},1)\)-admissible. If \(f=\chi_S\) for \(S\) with \(2<\mu(S)<\infty\) then \(\|f\|_\infty=1\) and
\[
    1\leq \liminf_{q\rightarrow\infty}\|f\|_{\Psi_q}\leq \limsup_{q\rightarrow\infty}\|f\|_{\Psi_q}\leq 2.
\]
On the other hand, we can show that \(\displaystyle\lim_{q\rightarrow\infty}\Psi_q^{-1}(t)\) does not exist for \(t\in (0,\frac{1}{2})\), meaning that
\begin{equation}\label{normlim}
    \lim_{q\rightarrow\infty}\|f\|_{\Psi_q}=\lim_{q\rightarrow\infty}\Psi_q^{-1}(\mu(S)^{-1})^{-1}
\end{equation}
does not exist. To see this, assume toward a contradiction that there is a \(t\in(0,\frac{1}{2})\) such that
\[
    \lim_{q\rightarrow\infty}\Psi_q^{-1}(t)=d.
\]

First we show that \(d\in(\frac{1}{2},1)\). Given \(\varepsilon>0\) we have for all large \(q\) that \(d-\varepsilon<\Psi_q^{-1}(t)<d+\varepsilon\) and \(\Psi_q(d-\varepsilon)<t<\Psi_q(d+\varepsilon)\). If \(d<\frac{1}{2}\) then we can choose \(\varepsilon\) small enough that \(d+\varepsilon\leq \frac{1}{2}\), meaning that \(\Psi_q(d+\varepsilon)\rightarrow0\) as \(q\rightarrow\infty\). Since \(t<\Psi_q(d+\varepsilon)\) for all large \(q\) this gives a contradiction for large \(q\). Likewise if \(d> 1\) then we can choose \(\varepsilon\) such that \(d-\varepsilon\geq 1\), meaning that \( \Psi_q(d-\varepsilon)\rightarrow\infty\) as \(q\rightarrow\infty\). Since \(\Psi_q(d-\varepsilon)<t<\frac{1}{2}\) this gives another contradiction, and it follows that \(\frac{1}{2}\leq d\leq 1\).

If \(d=\frac{1}{2}\) then we have \(\frac{1}{2}<d+\varepsilon< 1\) when \(\varepsilon\) is small and \(\Psi_q(d+\varepsilon)=\frac{1}{2}((\frac{1}{2}+\varepsilon)^q+(2\varepsilon)^{2+\sin q})\). Choosing \(\varepsilon<\frac{1}{2}\) so small that \((2\varepsilon)^{2+\sin q}\leq 2\varepsilon\leq t\) and then taking \(q\) so large that \((\frac{1}{2}+\varepsilon)^q\leq t\), we get \(\Psi_q(d+\varepsilon)\leq t\), a contradiction. Similarly, if \(d=1\) then \(\frac{1}{2}<d-\varepsilon<1\) for small \(\varepsilon\) and
\[
    \Psi_q(d-\varepsilon)=\frac{(1-\varepsilon)^q+(1-2\varepsilon)^{2+\sin q}}{2}\geq \frac{(1-2\varepsilon)^{2+\sin q}}{2}\geq \frac{(1-2\varepsilon)^{3}}{2}.
\]
Since \(t<\frac{1}{2}\), we can choose \(\varepsilon\) sufficiently small that \(\frac{(1-2\varepsilon)^{3}}{2}\geq t\), another contradiction. It follows that \(\frac{1}{2}<d<1\) and \(\frac{1}{2}<d-\varepsilon<d+\varepsilon<1\) for small \(\varepsilon\). Consequently
\[
    \Psi_q(d-\varepsilon)=\frac{(d-\varepsilon)^q+(2(d-\varepsilon)-1)^{2+\sin q}}{2}\;\;\textrm{ and }\;\;\Psi_q(d+\varepsilon)=\frac{(d+\varepsilon)^q+(2(d+\varepsilon)-1)^{2+\sin q}}{2}.
\]

Taking \(q=\frac{\pi}{2}+k\pi\) for odd \(k\in\mathbb{Z}\) gives \(t>\Psi_q(d-\varepsilon)=\frac{1}{2}((d-\varepsilon)^q+2(d-\varepsilon)-1)\), while using even \(k\) gives \(t<\Psi_q(d+\varepsilon)=\frac{1}{2}((d+\varepsilon)^q+(2(d+\varepsilon)-1)^{3})\). For \(\varepsilon\) small and \(q\) large we show that this is impossible. By convexity of the map \(t\mapsto t^q\) for \(q\geq 1\) we have
\[
    (d+\varepsilon)^q\leq \frac{1}{2}(d-\varepsilon)^q+\frac{1}{2}(d+3\varepsilon)^q<(d-\varepsilon)^q+(d+3\varepsilon)^q,
\]
and moreover, a straightforward calculation shows that
\[
    (2(d+\varepsilon)-1)^{3}=4d(2d-1)(d-1)+4\varepsilon(2+6d^2-6d+6d\varepsilon+2\varepsilon^2-3\varepsilon)+(2d-2\varepsilon-1).
\]
Note that \(4d(2d-1)(d-1)<0\) when \(\frac{1}{2}<d<1\). From the preceding estimates we get that
\[
    (d+\varepsilon)^q+(2(d+\varepsilon)-1)^{3}<(d-\varepsilon)^q+(d+3\varepsilon)^q+4d(2d-1)(d-1)+4\varepsilon(2+2\varepsilon^2+3\varepsilon)+(2d-2\varepsilon-1).
\]
Taking \(\varepsilon\) small and \(q\) large, we can ensure that \((d+3\varepsilon)^q+4d(2d-1)(d-1)+4\varepsilon(2+2\varepsilon^2+3\varepsilon)\leq 0\), and doing this gives the following contradiction: 
\[
    2t<(d+\varepsilon)^q+(2(d+\varepsilon)-1)^{3}\leq (d-\varepsilon)^q+2(d-\varepsilon)-1)<2t.
\]
Thus, the limit in \eqref{normlim} may not exist when the \(\delta\)-admissibility condition fails.

\section{Log-Bump Orlicz Norms}\label{log-sec}

Here we show that Theorem \ref{main} implies \cite[Theorem 1]{DCU-SR21}, which states that \eqref{infty:1} holds with \(\delta=1\) for a specific family of log-bump Young functions. Given $p\geq 1$, the log-bumps are of the form \(\Psi_q(t)=t^p\log(e-1+t)^q\) for \(q>0\), and the collection of all these bumps is a \(1\)-admissible family. Thus, \cite[Theorem 1]{DCU-SR21} follows from Theorem \ref{main} once we demonstrate that for $k>0$ the function \eqref{phicond} is non-decreasing on $[0,k]$ when $\Phi(t) = \Psi_{q_0}(t)$ for some $q_0>0$ and when $q$ is large.  To do this, we first assume that $q>q_0$ and for $t>0$ we define
\[
    F(t)=\frac{t}{\Psi_q^{-1}(\Psi_{q_0}(t))},
\]
so that \(F\) satisfies the equation \(\Psi_{q_0}(t)=\Psi_q(tF(t)^{-1})\). Recalling the form of $\Psi_q$ we see that
\begin{align}\label{inter-sr}
    F(t)^p\log(e-1+t)^{q_0}=\log\bigg(e-1+\frac{t}{F(t)}\bigg)^q
\end{align}
for \(t>0\). It follows from the definition above that \(F\) is continuous on \((0,\infty)\), and to extend \(F\) continuously to zero we observe that
\[
    \log(e-1)^{q_0}\lim_{t\rightarrow 0^+}F(t)^p=\log\bigg(e-1+\lim_{t\rightarrow 0^+}\Psi_q^{-1}(\Psi_{q_0}(t))\bigg)^q=\log(e-1)^q.
\]
Defining \(F(0)=\displaystyle\lim_{t\rightarrow 0^+}F(t)=\log(e-1)^\frac{q-q_0}{p}\) thus ensures that \(F\) is continuous on \([0,\infty)\).

\begin{lemma}\label{lemsr} If \(q>q_0\) then $F(t)>F(0)$ for every \(t>0\).
\end{lemma}

\begin{proof} Observe that if \(q>q_0\) then $\Psi_q(t)\geq \Psi_{q_0}(t)$ when $t\geq 1$, while \( \Psi_q(t)< \Psi_{q_0}(t)\) when \(0<t<1\). In the case $t\geq 1$, we use that $\Psi_{q}^{-1}$ is strictly increasing to see that \(t\geq \Psi_q^{-1}(\Psi_{q_0}(t))\), which implies that \(F(t)\geq 1>F(0)\). Likewise, \(0<t<1\) gives \(F(t)<1\) and by \eqref{inter-sr} we find
\[
    \log(e-1+t)^q<\log\bigg(e-1+\frac{t}{F(t)}\bigg)^q=F(t)^p\log(e-1+t)^{q_0}.
\]
Rearranging, we see that \(F(t)>\log(e-1+t)^\frac{q-q_0}{p}>F(0)\).
\end{proof}

Now fix $k>0$, set $I=[0,k]$ and let \(M=\sup_IF\). We show that \(F\) is injective on \(I\) when \(q\) is large enough. To this end, fix \(t_1\in I\), set \(c=F(t_1)\), and note that \(F(t_1)=F(0)\) if and only if \(t_1=0\). On the other hand, if \(t_1>0\) then \(c\in (F(0), M]\) by Lemma \ref{lemsr}.  More, from \eqref{inter-sr} we see that \(t_1\) is a fixed point of the map \(T_c:[0,\infty)\rightarrow \mathbb{R}\) defined by
\[
    T_c(t)=c\exp\big(c^\frac{p}{q}\log(e-1+t)^\frac{q_0}{q}\big)-c(e-1).
\]
Since $F(0)< c$, it is easy to see that \(T_c(0)>0\). Furthermore, a straightforward computation shows that \(T_c''(t)<0\) if and only if
\begin{equation}\label{Tconc}
    q_{0}c^{\frac{p}{q}}\log\left(e-1+t\right)^{\frac{q_{0}}{q}}<q\log\left(e-1+t\right)+q-q_{0}.
\end{equation}
For large \(q\),  this is achieved uniformly in $c\in[0,M]$.  To see why, choose $q\geq q_0$ so that
\[
     M^p\leq \bigg(\frac{q}{q_0}\bigg)^q\log(e-1)^{q-q_0}.
\]
Then \(c^\frac{p}{q}\leq \frac{q}{q_0}\log(e-1)^{1-\frac{q_0}{q}}\leq \frac{q}{q_0}\log(e-1+1)^{1-\frac{q_0}{q}}\) for each \(t\geq 0\), and this shows that \eqref{Tconc} holds for each \(t\geq 0\), and we see \(T_c(t)\) is strictly concave on $(0,\infty)$.

For \(q\) large as above, we find that $T_c(0)>0$ and $T_c$ is a continuous and strictly concave function on $(0,\infty)$.  Thus, $T_c(t)$ has a unique fixed point in $[0,\infty)$ and so it is $t_1$.  This gives $F$ injective on $I$ since $F(t_2)=c=F(t_1)$ shows that $t_1$ and $t_2$ are fixed points of $T_c(t)$. Since $F$ is a continuous, injective function on $I$, the Intermediate Value Theorem shows that $F$ is strictly monotone on $I$. Lastly, since Lemma \ref{lemsr} shows that $F(0)<F(t)$ for $t\in I$, we conclude that $F$ is strictly increasing on $I$ when $q$ is sufficiently large. Thus, with the hypotheses of Theorem \ref{main} verified, we have reproved \cite[Theorem 1]{DCU-SR21}.

\section{Necessity of Admissibility Conditions}\label{counters}

Finally, we show that Conditions \ref{a-admissible} and \ref{ab-admissible} are necessary for the norm limit to be related to the essential supremum of a function. This means that our admissibility conditions cannot be weakened in Theorem \ref{main} or Remark \ref{ab-rem}.

\begin{thm}\label{huh}
Let \(\{\Psi_q\}\) be a family of Young functions, let \(\Phi\) be a Young function for which \eqref{phicond} is non-decreasing, and assume that there exists \(f\in L^\Phi(X,\mu)\cap L^\infty(X,\mu)\) such that
\begin{equation}\label{boundss}
    0<\liminf_{q\rightarrow\infty}\|f\|_{\Psi_q}\leq \limsup_{q\rightarrow\infty}\|f\|_{\Psi_q}<\infty.
\end{equation}
Then  \(\{\Psi_q\}\) is \((\alpha,\beta)\)-admissible for some \(\beta>0\) and \(\alpha\geq 0\).
\end{thm}

\begin{proof}
First assume to the contrary that \(\Psi_q(t)\rightarrow\infty\) as \(q\rightarrow\infty\) for each fixed \(t>0\). Arguing as in the proof of Proposition \ref{ab-prop}, we conclude that \(\Psi_q^{-1}(t)\rightarrow 0\) as \(q\rightarrow\infty\) for each \(t>0\). Since the limit infimum in \eqref{boundss} is non-zero we have \(0<\|f\|_\infty\), and as in the proof of Theorem \ref{main} we can choose \(\varepsilon<\|f\|_\infty\) to see that
\[
    (\|f\|_\infty-\varepsilon)\liminf_{q\rightarrow\infty}\Psi_q^{-1}(\mu(\{x\in\Omega:|f(x)|\geq \|f\|_\infty-\varepsilon\})^{-1})^{-1}\leq \liminf_{q\rightarrow\infty}\|f\|_{\Psi_q}.
\]  
By definition of the essential supremum, \(\mu(\{x\in\Omega:|f(x)|\geq \|f\|_\infty-\varepsilon\})>0\), meaning that the limit on the left-hand side above diverges and \(\|f\|_{\Psi_q}\rightarrow\infty\) as \(q\rightarrow\infty\), contradicting the right-hand limit of \eqref{boundss}. Thus, there is a $t_0>0$ for which $\displaystyle\limsup_{q\rightarrow\infty}\Psi_q(t_0)<\infty$, and therefore\vspace{-0.1em}
\[
    \limsup_{q\rightarrow\infty}\Psi_q(t)<\infty
\]
holds for every $0\leq t\leq t_0$ since each Young function is strictly increasing. This means that \((\alpha,\beta)\)-admissibility holds for some \(\alpha\geq0\) and \(\beta>0\).

Similarly, suppose that \(\Psi_q(t)\rightarrow0\) as \(t\rightarrow\infty\) for each fixed \(t>0\), so that \(\Psi_q^{-1}(t)\rightarrow\infty\) by the argument of Proposition \ref{ab-prop}. Arguing as in Section \ref{proof} we have
\[
    \limsup_{q\rightarrow\infty}\|f\|_{\Psi_q}\leq\|f\|_\infty\limsup_{q\rightarrow\infty}\Psi_q^{-1}(\Phi(\|f\|_\infty))^{-1}=0,
\]
and once again this contradicts \eqref{boundss}. Since each \(\Psi_q\) is strictly increasing we conclude that \(\displaystyle\liminf_{q\rightarrow\infty}\Psi_q(t)>0\) for large \(t>0\). Thus there exists \(\alpha\leq\beta<\infty\) satisfying Condition \ref{ab-admissible}.
\end{proof}

In the case of \(\delta\)-admissibility, where \(\alpha=\beta=\delta\), the argument above shows that Theorem \ref{main} fails when \(\delta\) is not both positive and finite.

\begin{cor}
Let \(\{\Psi_q\}\) be a family of Young functions such that for every \(t>0\),
\[
    \lim_{q\rightarrow\infty}\Psi_q(t)=\infty\qquad(\textrm{resp. }\lim_{q\rightarrow \infty} \Psi_q(t)=0).
\]
If \(f\) satisfies the remaining hypotheses of Theorem \ref{main} then regardless of the value of \(\|f\|_\infty\),
\[
    \lim_{q\rightarrow \infty} \|f\|_{L^{\Psi_q}(X,\mu)}=\infty\qquad(\textrm{resp. }\lim_{q\rightarrow \infty} \|f\|_{L^{\Psi_q}(X,\mu)}=0).
\]
\end{cor}

To illustrate, if \(f=\chi_{[0,1]}\) and \(\Psi_q(t) = t^p\log(e+t)^q\) then \(\|f\|_{L^{\Psi_q}(\mathbb{R},dx)}=\Psi_q^{-1}(1)^{-1}\), but \(\{\Psi_q\}\) is not $\delta$-admissible for any $\delta>0$. A straightforward calculation shows that
\[
    \lim_{q\rightarrow\infty}\|f\|_{L^{\Psi_q}(\mathbb{R},dx)}\geq\liminf_{q\rightarrow\infty}\|f\|_{L^{\Psi_q}(\mathbb{R},dx)}=\liminf_{q\rightarrow\infty}\Psi_q^{-1}(1)^{-1}=\infty>1= \|f\|_{L^\infty(\mathbb{R},dx)}.
\]
Thus, if Condition \ref{a-admissible} fails then the induced Orlicz norms may not converge to \(\|f\|_{L^\infty(X,\mu)}\).

\bibliographystyle{alpha}

\end{document}